\documentclass{amsart}
\usepackage{enumerate}

\newtheorem{theorem}{Theorem}[section]
\newtheorem{lemma}[theorem]{Lemma}
\newtheorem{property}[theorem]{Property}
\newtheorem{assumption}{Assumption}[section]

\theoremstyle{definition}
\newtheorem{definition}[theorem]{Definition}
\newtheorem{example}[theorem]{Example}

\theoremstyle{remark}
\newtheorem{remark}[theorem]{Remark}

\numberwithin{equation}{section}

\begin{document}

\title{G-graded irreducibility and the index of reducibility}

\author{Cheng Meng}
\address{Department of Mathematics, Purdue University, 150 N. University Street, West Lafayette, IN 47907, USA}
\email{cheng319@purdue.edu}
\date{\today}


\begin{abstract}
Let $R$ be a commutative Noetherian ring graded by a torsionfree abelian group $G$. We introduce the notion of $G$-graded irreducibility and prove that $G$-graded irreducibility is equivalent to irreducibility in the usual sense. This is a generalization of Chen and Kim's result in the $\mathbb{Z}$-graded case. We also discuss the concept of the index of reducibility and give an inequality for the indices of reducibility between any radical non-graded ideal and its largest graded subideal.
\end{abstract}

\maketitle

\section{introduction}

Let $R$ be a commutative Noetherian ring, $M$ an $R$-module, and $N$ an $R$-submodule of $M$. It's known that $N$ has an irreducible decomposition, that is, $N$ is an intersection of irreducible submodules in $M$. When $R, M, N$ are all graded with respect to a torsionfree abelian group $G$, we can talk about $G$-graded irreducible submodules of $M$ and irreducible decomposition of $N$ in $M$ in the category of $G$-graded modules. It's natural to ask whether these two irreducibilities are the same. More precisely, we want to know whether graded irreducibility implies irreducibility in the nongraded sense. It's well known that irreducibility implies being primary; in [1, IV.3.3.5] we know being graded primary is the same as being primary. Chen and Kim proved in [3] that the two irreducibilities are the same in the $\mathbb{Z}$-graded case. In this paper we extend this result to the case of any $G$-grading where $G$ is a torsionfree abelian group. In particular, as a consequence, a $G$-graded irreducible decomposition is an irreducible decomposition in the usual sense, and both indexes of reducibility, defined for G-grading and in the usual sense, will be the same. Finally we estimate the indexes of reducibility of a nongraded ideal and its largest graded subideal. We prove one inequality in the radical case and show by example that it fails in general case.

In all the sections below we make the following assumptions unless otherwise stated:

\begin{assumption}$R$ is a commutative Noetherian ring, $M$ is a finitely generated $R$-module. When we say $R$ and $M$ are graded without mentioning the group used for grading, we are assuming that they are $G$-graded for a torsionfree abelian group $G$. The identity element of $G$ is denoted by 0.
\end{assumption}
The reason for these assumptions are as follows.

When $R$ is Noetherian and $N \subset M$ are Noetherian modules, we have a finite irredundant irreducible decomposition for N. So the index of reducibility defined below will make sense.

The torsionfree property is essential. In fact, in the $G$-graded case where $G$ has torsion, the definition of prime ideals, primary ideals, associated primes will be different. An example is the group algebra $k[\mathbb{Z}_2]=k[x]/(x^2-1)$. The ideal 0 is not a prime ideal; however it's a graded prime in the sense that if two homogeneous elements multiply to get 0 then one of them is 0. Also the associated primes $(x+1),(x-1)$ are all nongraded, so here we need a different definition for graded associated primes. Such definitions can be found in [6]. In the torsionfree case, a graded prime ideal is just a prime ideal that is graded; and the same holds for graded primary submodules and graded associated prime ideals.

\section{preliminaries}
We recall the following standard definitions.
\begin{definition}Let $(G,+)$ be an abelian group. A ring $R$ is said to be $G$-graded if there is a family of additive subgroups $R_g$ such that $R=\oplus_{g \in G}R_g$ and $R_gR_h \subset R_{g+h}$ for any $g,h \in G$. For a $G$-graded ring $R$, an $R$-module $M$ is $G$-graded if there is a family of additive subgroups $M_g$ such that $M=\oplus_{g \in G}M_g$ and $R_gM_h \subset M_{g+h}$ for any $g,h \in G$.
\end{definition}
\begin{definition}Let $R$ be a Noetherian ring, $M$ an $R$-module, $N$ an $R$-submodule of $M$. Then
\begin{enumerate}[(1)]
\item The submodule $N$ is called irreducible if whenever $N_1,N_2$ are two submodule of $M$ satisfying $N_1 \cap N_2 = N$ we have $N_1=N$ or $N_2=N$.
\item Suppose moreover that $R, M, N$ are $G$-graded. Then $N \subset M$ is called $G$-graded-irreducible, or simply graded-irreducible when $G$ is clear, if whenever $N_1,N_2$ are two $G$-graded submodule of $M$ satisfying $N_1 \cap N_2 = N$ we have $N_1=N$ or $N_2=N$.
\item The submodule $N$ is called primary if $M/N$ has only one associated prime. If this prime is $\mathfrak{p}$ we say $N$ is $\mathfrak{p}$-primary. The set of associated primes of a module $M$ is denoted by $Ass(M)$.
\end{enumerate}
\end{definition}
The above definitions hold for $N \subset M$ if and only if they hold for $0 \subset M/N$.

The following property is well known.
\begin{property} An abelian group is torsionfree if and only if it can be totally ordered.
\end{property}
\begin{proof} The if part is trivial. Now for the converse, if $G$ is torsionfree, we can embed $G$ into some $\mathbb{Q}$-vector space, order the basis element, and give the lexicographic order on the vector space and restrict this order to $G$.
\end{proof}

So now we can equip each torsionfree abelian group with a total order. We have the following property.

\begin{property}Let $R$ be a graded ring satisfying (1.1). Then
\begin{enumerate}
\item Let $\mathfrak{p}$ be a graded proper ideal in $R$ such that if $f, g$ are homogeneous elements in $R$, $fg \in \mathfrak{p}$, then $f \in \mathfrak{p}$ or $g \in \mathfrak{p}$. Then $\mathfrak{p}$ is a prime ideal.

\item Let $M$ be a graded $R$-module, $N$ be a graded submodule of $M$. Then every associated prime of $M/N$ is graded and is the annihilator of a homogeneous element. In particular, $N \subset M$ is primary if and only if $M/N$ has only one associated $G$-graded prime.

\item If $N$,$M$ are as in (2), then there is a graded primary decomposition.

\item If $N$,$M$ are as in (2) and $N \subset M$ is graded irreducible, then $M/N$ is graded primary, hence primary with a unique graded associated prime.
\end{enumerate}
\end{property}
\begin{proof}For (1), See [7, A.II.1.4]. For (2), See [7, A.II.7.3] or [5, Prop 3.12]. For (3), See [7, A.II.7.11], or [5, Ex 3.5]. (4) is a corollary of (3).
\end{proof}

The following definition comes from [7, A.I.4] ,[2, definition 1.5.13], and [3] in the $\mathbb{Z}$-graded case. In [2], [3] "$G$-graded local" is called "*local".
\begin{definition}A \emph{$G$-graded maximal ideal} of $R$ is a $G$-graded ideal $\mathfrak{m}$ which is maximal with respect to inclusion in all $G$-graded ideals properly contained in $R$. A $G$-graded ring $R$ is called \emph{$G$-graded local} if it has a unique $G$-graded maximal ideal. A \emph{$G$-graded field} is a $G$-graded ring $k$ such that all the nonzero homogeneous elements in $k$ are invertible.
\end{definition}
\begin{remark}A $G$-graded ideal $\mathfrak{m}$ is a $G$-graded maximal ideal if and only if $R/\mathfrak{m}$ is a $G$-graded field. In particular, if $k \neq 0$ is a $G$-graded ring, then it's a graded field if and only if it has only two $G$-graded ideals, namely 0 and $k$, if and only if 0 is a $G$-graded maximal ideal of $k$.
\end{remark}
\begin{definition}Let $R$ be a $G$-graded ring. For an ideal $I \subset R$ which is not necessarily $G$-graded, as in [2] and [3], we define $I^*$ to be the ideal of $R$ generated by all the homogeneous elements in $I$.
\end{definition}
\begin{remark}Assume (1.1). Since $G$ is torsionfree, Property 2.4(1) yields that $\mathfrak{p}^*$ is a graded prime ideal contained in $\mathfrak{p}$ when $\mathfrak{p}$ is a prime ideal of R. In particular, every $G$-graded maximal ideal $\mathfrak{m}$ in $R$ is a prime ideal of $R$, because $\mathfrak{m}$ is contained in some (not necessarily graded) maximal ideal $\mathfrak{n}$, and it follows that $\mathfrak{m}=\mathfrak{n}^*$ by definition. So a graded field $k$ must be a domain. Therefore it makes sense to talk about the rank of a $k$-module if $k$ is a graded field.
\end{remark}
\begin{definition}Let $R$ be $G$-graded, $M$ a $G$-graded $R$-module, and $\mathfrak{p}$ a $G$-graded prime of $R$. \emph{The homogeneous localization of $M$ at $\mathfrak{p}$}, denoted by $M_{(\mathfrak{p})}$, is $W^{-1}M$ where $W$ is the multiplicative set of all homogeneous elements not in $\mathfrak{p}$.
\end{definition}
If $R$ is graded, $\mathfrak{p}$ is a graded prime of $R$, then $R_{(\mathfrak{p})}$ is graded local.
\begin{lemma}Let $R$ be a Noetherian ring, $\mathfrak{p}$ a prime ideal of $R$, $M$ a finitely generated $R$-module, and $N$ an $R$-submodule of $M$ which is $\mathfrak{p}$-primary in $M$. Then $N$ is irreducible in $M$ if and only if $N_{\mathfrak{p}}$ is irreducible in $M_{\mathfrak{p}}$. Moreover, if $R,M,N,\mathfrak{p}$ are all $G$-graded, then $N$ is graded-irreducible in $M$ if and only if $N_{(\mathfrak{p})}$ is graded-irreducible in $M_{(\mathfrak{p})}$.
\end{lemma}
\begin{proof}See [3, Lemma 2]. The proof is for the $\mathbb{Z}$-graded case but can be applied to the $G$-graded case.
\end{proof}

\section{The structure of modules over graded fields}
It's well known that if $k$ is a field, then every vector space over $k$ is free. Here we prove a similar result when $k$ is a graded field.
\begin{definition}Let $R$ be a $G$-graded ring. We define the support of $R$, denoted by $Supp(R)$, to be $\{g \in G :R_g \neq 0\}$.
\end{definition}
If $R$ is a domain, then $Supp(R)$ is a subsemigroup of G. If $k$ is a graded field, then $Supp(k)$ is a subgroup of G.

The following two theorems have more general versions using the notions of strongly graded rings and graded division rings, see [7, A.I.3 and A.I.4]. We present explicit proofs in our particular case.

\begin{theorem}Let $G$ be a group, and $k$ a $G$-graded field. Then $k_0$ is a field, and $k_g$ $\cong$ $k_0$ as a $k_0$-vector space for any $g \in Supp(k)$.
\end{theorem}
\begin{proof}Every nonzero element in $k_0$ has an inverse in $k_0$, so $k_0$ is a field. Take any nonzero $u$ $\in k_g$. Then the multiplication by $u$ is a $k_0$ isomorphism of $k_0 \to k_g$ with an inverse which is the multiplication by $u^{-1}$.
\end{proof}
\begin{theorem}Let $G$ be a torsionfree abelian group, and $k$ a $G$-graded field. Then any $G$-graded $k$-module $M$ is free over $k$.
\end{theorem}
\begin{proof}Let $G'=Supp(k)$. Let $S$ be a set of representatives in $G$ of all the cosets in $G/G'$. Then $M = \oplus (M_g)_{g \in G} \cong \oplus_{s \in S } \oplus( (M_{sh})_{h \in G'})$ as $k$-module. So after shifting it suffices to prove $\oplus (M_h)_{h \in G'}$ is a free $k_0$-module for any s and any graded $k$-module $M$. Now $M_0$ is a $k_0$-vector space. Let $\{e_i, i \in I\}$ be a basis for $M_0$ and choose a basis $u_g$ in $k_g$ for each degree $g \in G$. Then in each degree, $u_g*e_i$ is a $k_0$-basis for $M_g$. This means that $M = \oplus_{i \in I} ke_i$. Hence $M$ is a free $k$-module.
\end{proof}
We want to restrict to the case where G is finitely generated using the Noetherian condition. We can do this in the case of the group algebra.
\begin{theorem}Let $G$ be a finite generated torsionfree abelian group, say $\mathbb{Z}^n$. Then every $G$-graded field k is isomorphic to $k_0[G']$ as a graded field, where $G'$ is the support of k.
\end{theorem}
\begin{proof} Since $G'$ is a subgroup of $G$, it is still a finitely generated torsionfree abelian group, say $\oplus_{i=1}^m \mathbb{Z}e_i$.For each i take a nonzero element $a_i$ in $k_{e_i}$. Then for any $h=n_1e_1+n_2e_2+\cdots+n_me_m \in G'$, $a_1^{n_1}a_2^{n_2}\cdots a_m^{n_m}$ is a nonzero element in $k_h$, thus $k_h=k_0*a_1^{n_1}a_2^{n_2}\cdots a_m^{n_m}$. This means that $k=k_0[a_1,a_2,...,a_m] \cong k_0[G']$.
\end{proof}
\begin{remark}The conclusion of the theorem above is not true in general for torsionfree abelian groups which are not finitely generated. In fact, we have to find a basis for all the nonzero components of the graded field or graded module; and there is no guarantee that one can find a collection of bases, labeled by the group, that is closed under multiplication if the group is not finitely generated. There are different isomorphic classes of such graded fields corresponding to the cohomology classes in $H^2(G',k_0^*)$; see [8, Ex 1.5.10].
\end{remark}
From the following theorem we see that if $k$ is a Noetherian group algebra, then its support is finitely generated. So we may assume $G$ is finitely generated in this case.
\begin{theorem}Let $k=k_0[G]$ be a group algebra over a field $k_0$, $G$ any abelian group, not necessarily torsionfree. Then $k$ is Noetherian if and only if $G$ is finitely generated.
\end{theorem}
\begin{proof}The if direction is obvious since in the finitely generated case the group algebra is the localization of a quotient of a polynomial ring over a field. Now suppose $G$ is not finitely generated. Consider any finitely generated ideal $I$. The generators live in finitely many degrees. Let H be the subgroup generated by these degrees, then $I$ must be $G/H$-graded for some finitely generated $H$.
Now consider the map $\pi: k \to k_0, \Sigma a_ig_i \to \Sigma a_i$. The kernel $J$ is an ideal in $k_0[G]$ generated by $(e_g-1)_{g \in G}$. If $H$ is a subgroup such that $J$ is $G/H$-graded then we must have $G=H$. Thus $J$ cannot be finitely generated.
\end{proof}

\section{G-graded irreducibility implies irreducibility}
In this section we prove our main result, that is, a graded irreducible submodule of a graded module is irreducible.
\begin{definition}Let $(R,\mathfrak{m},k)$ be a local ring or a graded-local ring. Let $M$ be an $R$-module. The socle of $M$, denoted by $soc(M)$, is $(0:_{M}\mathfrak{m})$. When $M$ is graded, $soc(M)$ is also graded. In both cases $soc(M)$ is a free k-module.
\end{definition}
\begin{lemma}Let $R$ be a Noetherian ring, $M$ a finitely generated $R$-module, and $N \subset M$ a submodule. Suppose:(1)$(R,\mathfrak{m},k)$ is local and $M/N$ is Artinian or (2)$R,M$ and $N$ are all $G$-graded, $(R,\mathfrak{m},k)$ is graded-local, $M/N$ is $\mathfrak{m}$-primary. Then $N \subset M$ is irreducible (resp. graded-irreducible) if and only if $soc(M/N)$ has rank 1.
\end{lemma}
\begin{proof}
We may assume $N=0$ after replacing $M$ with $M/N$. Now suppose the rank of $soc(M)$ is at least two. Then there exist a k-basis of $soc(M)$. Take the first two basis element: $e_1,e_2 \in soc(M) \subset M$ then $Re_1 \cap Re_2 =0$. So 0 is not irreducible. Now suppose the rank of $soc(M)$ is one, say, $soc(M)$ is $Re$ $\cong k$ as an R-module. Then for any $N_1,N_2 \in M$, $Ass (N_1) = Ass (N_2) = Ass (M) = \{\mathfrak{m}\}$ because $Ass (M)$ consists of one prime and $Ass(N_1),Ass(N_2)$ are nonempty subsets of $Ass (M)$. Now $soc(N_1) \neq 0,soc(N_2) \neq 0$ so we can take nonzero $e_1 \in soc(N_1), e_2 \in soc(N_2)$. They must all lie in $soc(M)= Re$. Now $k$ is a domain, hence 0 is an irreducible $k$-submodule of $k$, hence 0 is an irreducible $R$-submodule of the $R$-module $k$. So $Re_1 \cap Re_2 \neq 0$ in $Re \cong k$. So 0 is irreducible. In the graded case, just take all the modules to be graded and elements to be homogeneous.
\end{proof}
\begin{theorem}Let $R$ be $G$-graded, $M$ be a graded $R$-module, $N$ be a graded primary submodule of $M$, Then $N$ is graded irreducible if and only if $N$ is irreducible.
\end{theorem}
\begin{proof}We know that $M/N$ has a unique associated prime, denoted by $\mathfrak{p}$, and it's graded under assumption (2). Also, we may assume $N=0$ after replacing $M$ with $M/N$. The "if" direction is trivial. Now let 0 be graded irreducible in $M$. Then 0 is an $R_{(\mathfrak{p})}$-submodule in $M_{(\mathfrak{p})}$ which is graded irreducible by Lemma 2.10. Then $soc(M_{(\mathfrak{p})}) \cong R_{(\mathfrak{p})}/\mathfrak{p}R_{(\mathfrak{p})}$. Now $(0 :_{M_{(\mathfrak{p})}} \mathfrak{m}_{(\mathfrak{p})})_\mathfrak{p}=(0 :_{M_{\mathfrak{p}}} \mathfrak{m}_{\mathfrak{p}})$ because $\mathfrak{m}$ is finitely generated. So $soc(M_{\mathfrak{p}}) \cong R_{\mathfrak{p}}/{\mathfrak{p}}R_{\mathfrak{p}}$. This means 0 is irreducible in $M_{\mathfrak{p}}$. So 0 is irreducible in $M$ by Lemma 2.10.
\end{proof}
We have proved that graded-irreducibility is the same as being graded and irreducible. Now we give the following definitions  from [3] and [4], generalized to the $G$-graded case. In [3] they are called the index of irreducibility and denoted by $r_M(N)$ (resp. $r_M^g(N)$).
\begin{definition}Let $N \subset M$ be $R$-modules. Then
\begin{enumerate}
\item The index of reducibility of $N$ in $M$ is $ir_M(N)=min\{r:N=\cap_{i=1}^r N_i, N_i$ irreducible $R$-submodules of $M\}$.
\item When $M$ and $N$ are graded, the graded index of reducibility of $N$ in $M$ is $ir_M^g(N)=min\{r:N=\cap_{i=1}^r N_i, N_i$ graded-irreducible $R$-submodules of $M\}$.
\end{enumerate}
\end{definition}
When $M$ is clearly understood we simply denote them by $ir(N)$ (resp. $ir^g(N)$).
Here is the $G$-graded version of [3, Theorem 7]. The proof is identical.
\begin{theorem}Let $R$ be a $G$-graded ring, $M$ a $G$-graded module where $G$ is abelian but not necessarily torsionfree. Then the following are equivalent.
\begin{enumerate}
\item Every graded-irreducible submodule of $M$ is irreducible.

\item For every graded submodule $N$ of $M$, $ir_M(N)=ir^g_M(N)$.

\item Every graded submodule $N$ of $M$ is a finite intersection of irreducible graded submodules of $M$.
\end{enumerate}
In particular these equivalent conditions all hold if $G$ is torsionfree abelian.
\end{theorem}

\begin{theorem}Let $(R,{\mathfrak{m}},k)$ be graded local, $N \subset M$ are graded $R$-modules such that $M/N$ is ${\mathfrak{m}}$-primary. Then $ir_M(N) = rank_k soc(M/N)$.
\end{theorem}
\begin{proof}Localizing at ${\mathfrak{m}}$. We have $ir_M(N) = rank_{k_{\mathfrak{m}}} soc(M/N)_{\mathfrak{m}}$. Notice that the rank of $soc(M/N)$ will not change after localizing.
\end{proof}

\section{The relation between the index of reducibility of $I$ and $I^*$}
Let $R$ be a graded ring and $I$ be an ideal of $R$ which is not necessarily graded. We want to compare $ir_R(I)$ and $ir_R(I^*)$. Let's consider a special case: $G=\mathbb{Z}$ and $R$ be the coordinate ring of a cone $C$ in an affine variety $\mathbb{A}^n$, then $R$ is $G$-graded. In this case the operation $I \to I^*$ has a geometric interpretation. Suppose $I$ is a radical ideal corresponding to a closed subset $X$ in $C$, and $X$ is not supported at the origin. That is, $X=V(I)$ is the vanishing set of $I$. There is a natural projection $\pi: \mathbb{A}^n-\{(0)\} \to \mathbb{P}^{n-1}$. $\pi$ restricts to two maps: $C-\{0\} \to \mathbb{P}^{n-1}$ and $X-\{0\} \to \mathbb{P}^{n-1}$. Then $I^*$ is a radical ideal, and its vanishing set is $\overline{\pi(X-\{0\})}$ in $\mathbb{P}^{n-1}$, because the maximal homogeneous ideal in $I$ corresponds to the minimal closed subset containing $\pi(X-\{0\})$.

For a morphism $f:Y \to Y'$ between varieties, if $Z \subset Y$ is an irreducible closed subset, then $\overline{f(Z)}$ is also irreducible. So if all the irreducible components are reflected in the variety as a set, $ir_R(I)$ should be greater or equal to $ir_R(I^*)$, because every irreducible component of $I$ or $V(I)$  map to an irreducible subset contained in $V(I^*)$. The equality holds if and only if different irreducible components do not collapse to contain each other. Let $Min(J)$ denote the minimal prime over an ideal $J$, and $|S|$ denote the cardinality of a set $S$. We have the following theorem:
\begin{theorem}Let $R$ be a $G$-graded ring where $G$ is torsionfree abelian. Let $I$ be an $R$-ideal which is radical but not $G$-graded. Then $ir(I)\geq ir(I^*)$. The equality holds if and only if the *-map induces a bijection between $Min (I)$ and $Min (I^*)$.
\end{theorem}
\begin{proof}The irreducible decomposition of I is $I= \cap (p_i)_{i \in Min(I)}$ and this decomposition is irredundant. So $ir(I)=|Min(I)|$. Now taking star commutes with intersection, so $I^*= \cap (p_i^*)_{i \in Min(I)}$. Note that $I^*$ is also radical so $ir(I^*)=|Min(I^*)|$. After deleting some $p_i^*$ it becomes an irredundant irreducible decomposition. So $ir(I) \geq ir(I^*)$. The equality holds if and only if no prime ideal is deleted, so those $p_i^*$ are just all the minimal primes of $I^*$, so the *-map induces a bijection between $Min (I)$ and $Min (I^*)$.
\end{proof}
In general there is no control on the difference between $ir(I)$ and $ir(I^*)$.
\begin{example}Let $R=k[x,y]$ which is $\mathbb{Z}$-graded. Let $I= (x,(y-a_1)(y-a_2)\cdots(y-a_r))$ where $a_1,a_2,...,a_r$ are pairwise distinct and all nonzero elements in $k$. Then $I^*=(x)$. In this case we see that $I$ has r components which collapse to become one component of $I^*$.
\end{example}

In general, there is no fixed inequality between $ir(I)$ and $ir(I^*)$. Here are two examples where $I$ and $I^*$ are both ${\mathfrak{m}}$-primary for a graded prime ideal ${\mathfrak{m}}$.
\begin{example}Let $R=k[x,y]$ be $G=\mathbb{Z}$ graded where $k$ is a field. Let $I=(x^2,xy,y^3,x-y^2)$, then $I^*=(x^2,xy,y^3)$. They are both $m=(x,y)$-primary. $R/m=k$. $soc(I)=kx,soc(I^*)=kx+ky^2$. So $ir(I)=1<2=ir(I^*)$.
\end{example}
\begin{example}Let $R,m,G$ be as above, $I=(x^4,x^2y^2,y^4,x^3y-y^3x)$. $I^*=(x^4,x^2y^2,y^4)$. They are still $m$-primary. Now $soc(I^*)=kx^3y+kxy^3, soc(I)=kx^3y+k(x^3-xy^2)+k(x^2y-y^3)$. So $ir(I)=3>2=ir(I^*)$.
\end{example}

\section*{Acknowledgements}
The author would like to thank Professor Giulio Caviglia for introducing this problem and showing a proof in the $\mathbb{Z}$-graded case. The author would like to thank Professor William Heinzer for reading an early draft.

\end{document}